\documentclass[11pt,a4paper]{amsart}
\usepackage[utf8]{inputenc}
\usepackage[T1]{fontenc}
\usepackage[UKenglish]{babel}
\usepackage[a4paper,margin=1in]{geometry}

\usepackage{amsmath,amsthm,amsfonts,amssymb,accents}
\usepackage{mathrsfs,dsfont}
\usepackage{graphicx}
\usepackage{url}
\usepackage{verbatim}
\usepackage{xcolor}

\usepackage{hyperref}

\usepackage{marginnote}

\usepackage{soul} 
\setstcolor{magenta}

\makeatletter
\def\namedlabel#1#2{\begingroup
    #2%
    \def\@currentlabel{#2}%
    \label{#1}\endgroup
}
\makeatother

\theoremstyle{plain}
\newtheorem{theorem}{Theorem}[section]
\newtheorem{corollary}[theorem]{Corollary}
\newtheorem{lemma}[theorem]{Lemma}
\newtheorem{proposition}[theorem]{Proposition}

\theoremstyle{definition}
\newtheorem{remark}[theorem]{Remark}
\newtheorem{example}[theorem]{Example}
\newtheorem{definition}[theorem]{Definition}

\numberwithin{equation}{section}

\renewcommand\labelenumi{\textup{\alph{enumi})}}

\renewcommand\theenumi\labelenumi
\makeatletter\renewcommand{\p@enumii}{}\makeatother 

\renewcommand{\leq}{\leqslant}

\renewcommand{\geq}{\geqslant}

\newcommand{\R}{\mathds{R}}
\newcommand{\N}{\mathds{N}}

\newcommand{\I}{\mathds{1}}

\newcommand{\re}{\mathop{\mathrm{Re}}}

\begin{document}

\title[On directional convolution equivalent densities]{On directional convolution equivalent densities}
\author{Kamil Kaleta and Daniel Ponikowski}

\address{Kamil Kaleta and Daniel Ponikowski \\ Faculty of Pure and Applied Mathematics,
  Wroc{\l}aw University of Science and Technology,
  Wybrze{\.z}e Wyspia{\'n}\-skie\-go 27,
  50-370 Wroc{\l}aw, Poland}
	\thanks{Research supported by National Science Centre, Poland, grant no.\ 2019/35/B/ST1/02421}
\email{kamil.kaleta@pwr.edu.pl, d.ponikowski11@gmail.com}

\maketitle

\begin{abstract}
We propose a definition of directional multivariate subexponential and convolution equivalent densities
and find a useful characterization of these notions for a class of integrable 
and almost radial decreasing functions. We apply this result to show that the 
density of the absolutely continuous part of the compound Poisson measure 
built on a given density $f$ is directionally convolution equivalent and inherits its 
asymptotic behaviour from $f$ if and only if $ f$ is directionally convolution equivalent.
We also extend this characterization to the densities of more general infinitely divisible distributions on 
$\R^d$, $d \geq 1$, which are not pure compound Poisson.


\bigskip
\noindent
2010 {\it MS Classification}: 60E05, 60G50, 60G51, 26B99, 62H05.\\
  
\end{abstract}

\section{Introduction and statement of results}

\subsection*{Motivation, basic definition and related previous results}
In the recent paper \cite{KSz3} the spatial asymptotics at infinity for the heat kernels of non-local pseudo-differential operators (or transition densities of jump L\'evy processes) have been investigated.  
The key step in that paper was based on the observation, which can be summarized as follows. Let $f$ be a density of the finite measure on $\R^d$, $d \geq 1$, which is uniformly comparable to the radial decreasing function (see the definition below) and satisfies
\begin{equation*} 
        \frac{f(t \theta-y)}{f(t\theta)} \stackrel{t \to \infty}{\longrightarrow} e^{\gamma(\theta \cdot y)}, \quad y \in \R^d,
 \end{equation*}
for a given $\theta \in \mathds{S}^{d-1}$ (the unit sphere in $\R^d$) and $\gamma \geq 0$,
and let the map $K_f:(0,\infty) \to (0,\infty]$ be defined by
$$
K_f(r):=\sup_{|x| \geq 1} \int_{|x-y|>r \atop |y| >r} \frac{f(x-y)f(y)}{f(x)} dy.
$$  
Then the following implication holds: if 
\begin{equation} \label{eq:K0_first}
\lim_{r \to \infty}K_f(r) = 0,
\end{equation}
then
		\begin{equation} \label{eq:c2_0}
        h(\theta): = \int\limits_{\R^d} e^{\gamma(\theta \cdot y)}f(y)dy \textless \infty
    \quad \text{and} \quad
        \lim\limits_{t \to \infty}\frac{f^{n*}(t\theta)}{f(t\theta)} = n h(\theta)^{n-1}, \quad n \in \N,
    \end{equation}
where $f^{n\star}$ denotes the $n$-fold convolution of $f$. In \cite{KSz3} this implication was used to derive the asymptotics for densities of compound Poisson measures and more general convolution semigroups of measures on $\R^d$. Although the condition \eqref{eq:K0_first} is very effective in applications, the question about its optimality remains open. This was a primary motivation for our present considerations.  

The main goal of this paper is to show that \eqref{eq:c2_0} also implies \eqref{eq:K0_first}, i.e.\ the condition \eqref{eq:K0_first} in fact fully characterizes the convergence in \eqref{eq:c2_0} (Theorem \ref{th:main} and Corollary \ref{cor:characterization}). We also want to look at \eqref{eq:c2_0} from a slightly different perspective -- this condition can be seen as a multivariate directional variant of the convergence known from the theory of subexponential and convolution equivalent densities on $(0,\infty)$, see Kl\"uppelberg \cite{Klup} and references therein. We define formally the class of directional convolution equivalent densities $\mathscr{S}_d(\theta, \gamma)$ on $\R^d$, $d \geq 1$, and analyse some of its properties.
Our second main result says that the density $f$ which is uniformly comparable to a radial decreasing function is in $\mathscr{S}_d(\theta, \gamma)$ if and only if the density of the corresponding compound Poisson measure belongs to $\mathscr{S}_d(\theta, \gamma)$  and inherits its 
asymptotic behaviour from $f$ (see Theorem \ref{th:poisson_characterization} for precise statement). We also extend this characterization to the densities of fairly general infinitely divisible distributions on $\R^d$, $d \geq 1$, which are not pure compound Poisson, including those built on infinite L\'evy measures (Theorem \ref{th:inf_div_character}). The argument is based on \eqref{eq:K0_first}. Our new Theorem \ref{th:main} is completely critical for the proof of these results.

Throughout we consider non-negative, non-zero functions $f \in L^1(\R^d)$ (taken with respect to the Lebesgue measure) that we call \emph{densities}. We propose the following definition.

\begin{definition}[\textbf{Directional subexponential and convolution equivalent densities}] \label{def:main_def} 
{\rm \ \newline
Let $f \in L^1(\R^d)$, $f \geq 0$ and let $\theta \in \mathds{S}^{d-1}$ and $t_0>0$ be such that $f(t\theta) >0$, for $t \geq t_0$. 
We say that $f$ belongs to the class $\mathscr{S}_d(\theta, \gamma)$ with $\gamma \geq 0$ if
    \begin{equation} \label{eq:c1}
        \frac{f(t \theta-y)}{f(t\theta)} \stackrel{t \to \infty}{\longrightarrow} e^{\gamma(\theta \cdot y)}, \quad  \text{for every} \ \ y \in \R^d;
    \end{equation}
		\begin{equation} \label{eq:c2}
        h(\theta): = \int\limits_{\R^d} e^{\gamma(\theta \cdot y)}f(y)dy \textless \infty
    \quad \text{and} \quad
        \frac{f^{2\star} (t\theta)}{f(t\theta)} \stackrel{t \to \infty}{\longrightarrow} 2h(\theta).
    \end{equation}		
The class $\mathscr{S}_d(\theta, 0)$ will be called the \emph{class of subexponential densities in the direction $\theta \in \mathds{S}^{d-1}$}. More generally, $\mathscr{S}_d(\theta, \gamma)$, $\gamma \geq 0$, will be called the \emph{class of convolution equivalent densities with parameter $\gamma$ in the direction $\theta \in \mathds{S}^{d-1}$}. 
}
\end{definition} 

The study of the subexponential and convolution equivalent distributions on half-line dates back to papers by Chistyakov \cite{Chist}, Athreya and Ney \cite{AN}, Chover, Ney and Wainger \cite{CNW73,CNW73b}) (see also the monographs \cite{EKM} by Embrechts, Kl\"uppelberg and Mikosch, and \cite{FKZ} by Foss, Korshunov and Zachary, and further references therein). They were mainly investigated in connection with branching processes, renewal and queueing theory, random walks, and infinitely divisible distributions. The key applications were based on an analysis of the asymptotic behaviour of distributions of the so-called random sums which were constructed for given initial distributions. A key example of such a random sum is that of Poissonian type (the compound Poisson measure). A basic result obtained for this example states that the compound Poisson distribution inherits its asymptotics at infinity from the initial distribution if and only if this distribution is convolution equivalent, see Embrechts, Goldie and Veraverbeke \cite{EGV} for the subexponential case, Embrechts and Goldie \cite{EG} for general convolution equivalent distributions, and Cline \cite{Cl} for an extension to more general random sums. More recently, similar asymptotic results have been obtained for general one-dimensional infinitely divisible laws by Sgibnev \cite{Sg}, Pakes \cite{Pa1, Pa2}, Shimura and Watanabe \cite{W05}, Watanabe \cite{Wat2} and Watanabe and Yamamuro \cite{W10, WY} (see also Yakymiv \cite{Y}).
The papers which are most closely related to our present contributions are works by Kl\"uppelberg \cite{Klup}, Finkelshtein and Tkachov \cite{FT} and Watanabe \cite{W20} which develop the theory of convolution equivalent and subexponential \emph{densities} on the half-line and the line. Note, however, that the densities have been considered much earlier by Chover, Ney and Wainger \cite{CNW73,CNW73b}. See also Foss and Zachary \cite{FZ}, Korshunov \cite{Kor} and Wang and Wang \cite{WW} for applications.  

Interestingly, there is no canonical definition of subexponential and convolution equivalent distributions in $\R^d$, $d >1$. Depending on applications, one can find at least three different definitions in the literature, see Cline and Resnick \cite{CR92}, Omey \cite{Omey} (see also Omey, Mallor and Santos \cite{OMS}) and recent paper by Samorodnitsky and Sun \cite{SS}. We refer the reader to an excellent overview of the existing multivariate theory given in \cite{SS}. 
Recently, the results from \cite{Omey,OMS} have been used by Knopova \cite{Knop} to prove the so-called `bell-like estimates' for the densities of L\'evy processes in $\R^d$, and by Knopova and Palmowski \cite{KP} to establish the asymptotics of the potential of some killed Markov proceses (with applications to insurance models).
The authors of \cite{SS} also applied their approach to compute the asymptotic behaviour of the ruin probability in a certain insurance model. 

We are not aware of any previous papers dealing directly with convolution equivalent densities in $\R^d$, $d>1$. Finkelshtein and Tkachov introduced in \cite{FT}
the class of regular subexponential densities on the line and used it to find an analogue of Kesten's bound for radial functions on $\R^d$, $d \geq 1$. The authors of that paper remark that even in the radial case the study of the asymptotic behaviour of convolutions of multivariate densities is a rather difficult and open problem (see discussion on p.\ 375 in \cite{FT}; see also comments preceding Theorem 3 in \cite{Knop}). It seems to be due to the fact that although the convolution of two multivariate radial functions is still radial, its analysis cannot be directly reduced to the study of univariate functions. In this paper, we propose a directional approach. The results in \cite{KSz3} and our present contributions can be seen as a first attempt to this problem.

Observe that the above definition allows one to consider the convolution equivalence in any generalized cone in $\R^d$ centered at $0$. Let $E$ be an arbitrary non-empty subset of the unit sphere $\mathds{S}^{d-1}$. By $\Gamma_E$ we denote the generalized cone in $\R^d$ centered at $0$ and based on $E$, i.e. 
$$
\Gamma_E = \left\{x \in \R^d \setminus \left\{0\right\}: \frac{x}{|x|} \in E \right\}.
$$
If $f \in \mathscr{S}_d(\theta, \gamma)$ for every $\theta \in E$, then it can be seen as a \emph{convolution equivalent density in cone $\Gamma_E$.} Specifically, the choice $\Gamma_E = (0,\infty)^d$ links our framework with that of \cite{Omey}. We expect that such a flexibility will allow one to find various applications for directional convolution equivalence, similar to those in \cite{KP} and \cite{SS}. Let us also remark that Definition \ref{def:main_def} naturally generalizes the definition for the densities on the real line (cf.\ e.g.\ \cite[Definition 3]{Klup}).
The set $(0,\infty)$ is an example of a generalized cone in $\R$. In general, in this case one can consider $E \subset \mathds{S}^{0} = \left\{-1,1\right\}$.

\subsection*{Presentation of our results and further discussion}

In this paper we develop the concept of directional convolution equivalence for a class of multivariate densities which are uniformly comparable to radial decreasing functions. 
We say that the density $f$ is \emph{almost radial decreasing} if there exists a decreasing function $g : (0, \infty) \rightarrow (0, \infty)$ such that $f(x) \asymp g(|x|)$, $x \in \R^d \setminus \left\{0\right\}$ (we say that $g:(0,\infty) \to \R$ is decreasing if $g(s) \geq g(t)$ for all $s, t >0$ such that $s \leq t$; we never require strict monotonicity). The function $g$ is then called a \emph{profile of the density} $f$. The notation $f(x) \asymp h(x)$, $x \in D$, means that there exists
a constant $c \geq 1$ such that
\begin{equation*}
    \frac{1}{c} h(x) \leqslant f(x) \leqslant c h(x), \quad x \in D.
\end{equation*}

As we mentioned above, a sufficient condition for directional convolution equivalence, given in terms of $K_f$, follows from results in \cite{KSz3}.

\begin{proposition}{\cite[Corollary 4 and Lemma 6(a)]{KSz3}}\label{prop:sufficient}
Let $f$ be an almost radial decreasing density and let \eqref{eq:c1} holds with some $\theta \in \mathds{S}^{d-1}$ and $\gamma \geq 0$.  Then the condition 
\begin{equation} \label{eq:K0}
\lim_{r \to \infty}K_f(r) = 0
\end{equation}
implies that $h(\theta) < \infty$ and the convergence
\begin{equation} \label{eq:conv_conv}
\lim\limits_{t \to \infty}\frac{f^{n*}(t\theta-y)}{f(t\theta)}=e^{\gamma(\theta \cdot y)} n h(\theta)^{n-1}, \quad y \in \R^d, \ \ n \in \N,
\end{equation} 
holds. In particular, $f \in \mathscr{S}_d(\theta, \gamma)$.
\end{proposition} 

Our first main result in this paper states that also the converse implication is true. 
We need here some additional notation:
$$
K_f^A(r) := \sup_{|x|\geq A} \int_{\substack{|x-y| \textgreater r \\|y| \textgreater r}}  \frac{f(x-y)f(y)}{f(x)}dy, \quad r \geq A.
$$
Clearly, $K^1_f(r) = K_f(r)$, $r \geq 1$.

\begin{theorem} \label{th:main}
Let the density $f$ be such that there exist $A \geq 1$ and a decreasing function $g:[A,\infty) \to (0,\infty)$ satisfying $f(x) \asymp g(|x|)$, $|x| \geq A$. Then the conjunction of \eqref{eq:c1} and \eqref{eq:c2} implies 
\begin{equation} \label{eq:KA0}
\lim_{r \to \infty}K^A_f(r) = 0.
\end{equation}
In particular, if $f$ is almost radial decreasing, then the conjunction of \eqref{eq:c1} and \eqref{eq:c2} implies \eqref{eq:K0}.
\end{theorem}

The following corollaries, which give the full characterization, are a direct consequence of Proposition \ref{prop:sufficient} and Theorem \ref{th:main}.

\begin{corollary}\label{cor:characterization}
Let $f$ be an almost radial decreasing density and let \eqref{eq:c1} holds with some $\theta \in \mathds{S}^{d-1}$ and $\gamma \geq 0$.  Then $f \in \mathscr{S}_d(\theta, \gamma)$ if and only if $\lim_{r \to \infty}K_f(r) = 0$.
\end{corollary}

We illustrate this result by examples in Section \ref{sec:ex}. In Example \ref{ex:ex1} we prove that if the density $f$ satisfying \eqref{eq:c1} (with $\gamma = 0$) has a doubling property, then we always have $\lim_{r \to \infty}K_f(r) = 0$ and, in consequence, $f \in \mathscr{S}_d(\theta,0)$. In Example \ref{ex:ex2} we analyse the density $f$ on $\R^d$, $d \geq 1$, with the profile 
$$
g(r) = e^{-m r} (1 \vee r)^{-\beta}, \quad m > 0, \ \ \beta \geq 0,
$$
which fulfils \eqref{eq:c1}. We show that $f \in \mathscr{S}_d(\theta, m)$ if and only if $\beta > (d+1)/2$ (Proposition \ref{prop:ex2}), cf.\ \cite[Example to Corollary 2.2]{Klup}. These examples demonstrate well the utility of \eqref{eq:KA0}. The main advantage of the approach based on the map $K_f(r)$ is that instead of proving directly the difficult convergence in \eqref{eq:c2} we can just try to establish the estimate which leads to \eqref{eq:KA0}. 

The next corollary should be compared with analogous one-dimensional results in \cite[Lemma 3.1 (c)]{Klup} and \cite[Theorem 1.1]{FT}. 

\begin{corollary}\label{cor:characterization_2}
Let $f$ be an almost radial decreasing density satisfying \eqref{eq:c1} with some $\theta \in \mathds{S}^{d-1}$ and $\gamma \geq 0$.  Then $f \in \mathscr{S}_d(\theta, \gamma)$ if and only if $h(\theta) < \infty$ and the convergence \eqref{eq:conv_conv} holds.
\end{corollary}

We now present our theorems for densities of compound Poisson and more general infinitely divisible distributions. Their proofs are based on an application of Theorem \ref{th:main}.

Let $f \in L^1(\R^d)$ and let $P_{\lambda}(dx)$, $\lambda>0$, be the corresponding compound Poisson measure on $\R^d$ which is given by 
$$
P_{\lambda}(dx) = e^{-\lambda\left\|f\right\|_1} \delta_0(dx) + p_{\lambda}(x)dx, 
$$
where $\delta_0(dx)$ is the Dirac delta measure and
\begin{align}\label{eg:poiss_dens}
p_{\lambda}(x) = e^{-\lambda\left\|f\right\|_1} \sum_{n=1}^{\infty} \frac{\lambda^n}{n!} f^{n \star}(x).
\end{align}
 
The following theorem is our second main result. It can be seen as a multivariate Poissonian variant of the one-dimensional result in \cite[Theorem 3.2]{Klup}. 

\begin{theorem}\label{th:poisson_characterization}
Let $f \in L^1(\R^d,dx)$, $f \geq 0$, be an almost radial decreasing density such that \eqref{eq:c1} holds with some $\theta \in \mathds{S}^{d-1}$ and $\gamma \geq 0$. Then the following statements are equivalent.
\begin{itemize}
\item[(a)] $f \in \mathscr{S}_d(\theta, \gamma)$;
\item[(b)] We have
\begin{align} \label{eq:char_exp}
\int_{\R^d} e^{\gamma(\theta \cdot y)} f(y) dy < \infty
\end{align}
and for every $\lambda>0$, 
\begin{align} \label{eq:char_conv}
\lim_{t \to \infty} \frac{p_{\lambda}(t\theta - y)}{\lambda f(t\theta)} = \exp{\left(\gamma(\theta \cdot y) + \lambda \int\left(e^{\gamma (\theta \cdot z)} - 1\right) f(z)dz\right)}, \quad y \in \R^d;
\end{align}
\item[(c)] For some (every) $\lambda>0$, $p_{\lambda} \in \mathscr{S}_d(\theta, \gamma)$ and $\limsup_{|x| \to \infty} \frac{p_{\lambda}(x)}{f(x)} < \infty$.
\end{itemize}
\end{theorem}

We also give a similar result for the densities of more general infinitely divisible distributions on $\R^d$, $d \geq 1$, which are not pure compound Poisson.
Let $P_{\lambda}(dx)$, $\lambda>0$, be a measure on $\R^d$ which is (uniquely) determined by its Fourier transform 
\begin{align} \label{eq:id_measure}
\mathcal{F}(P_{\lambda})(\xi)=\int_{\R^d} e^{i \xi \cdot y}P_{\lambda}(dy)=\exp(- \lambda \psi(\xi)), \quad \xi \in \R^d,
\end{align}
where the exponent $\psi$ is given by the L\'evy--Khintchine formula
$$
  \psi(\xi) =  - i \xi \cdot b + \xi \cdot A \xi  + \int_{\R^d \setminus \left\{0\right\}} \left(1 - e^{i \xi \cdot y} + i \xi \cdot y \I_{B(0,1)}(y)\right)\nu(dy) , \quad \xi \in \R^d.
$$
Here $b \in \R^d$, $A$ is a symmetric non-negative definite $d \times d$ matrix and $\nu$ is a measure on $\R^d \setminus \left\{0\right\}$ such that $\int_{\R^d \setminus \left\{0\right\}} (1 \wedge |x|^2) \nu(dx) < \infty$ \cite{Sato}. We also denote 
$$
\phi(\xi) = \int_{\R^d \setminus \left\{0\right\}} \left(1 - e^{i \xi \cdot y} + i \xi \cdot y \I_{B(0,1)}(y)\right)\nu(dy), \quad \xi \in \R^d
$$
 for shorthand.

We have to impose some mild regularity assumptions on $A$ and $\nu$. We are able to proceed in three general, disjoint cases:

\smallskip
\noindent
(1) $\inf_{|\xi|=1} \xi \cdot A \xi > 0 $ and $\nu$ is a finite measure;

\smallskip
\noindent
(2) $\inf_{|\xi|=1} \xi \cdot A \xi > 0$, $\nu(\R^d \setminus \left\{0\right\})=\infty$ and there is $\lambda_0>0$ such that $\int_{\R^d} e^{-\lambda_0 \re \phi(\xi)} |\xi| d \xi < \infty$;

\smallskip
\noindent
(3) $A \equiv 0$, $\nu(\R^d \setminus \left\{0\right\})=\infty$ and there is $\lambda_0>0$ such that $\int_{\R^d} e^{-\lambda_0 \re \phi(\xi)} |\xi| d \xi < \infty$.

\smallskip

To shorten the notation below, it is convenient to define $\Lambda=(0,\infty)$ in case (1), and $\Lambda=[\lambda_0,\infty)$ in cases (2) and (3) above.
Under each of these three conditions, for every $\lambda \in \Lambda$, the measures $P_{\lambda}$ are absolutely continuous with respect to Lebesgue measure and the corresponding densities $p_{\lambda}$ are bounded and continuous on $\R^d$. Indeed, we always have $\int_{\R^d} e^{-\lambda \re \psi(\xi)} d \xi < \infty$ for $\lambda \in \Lambda$ and $p_{\lambda}$ can be expressed by the Fourier inversion formula. In case (2) the same is true for every $\lambda>0$, but due to singularity of the L\'evy measure $\nu$ near zero, we analyse the asymptotic behaviour of $p_{\lambda}$ for $\lambda \in \Lambda = [\lambda_0, \infty)$ only. Observe also that we do not need to consider separately the case when $A \equiv 0$, $b \neq 0$ and $\nu$ is a non-zero finite measure (the compound Poisson distribution with drift). The result for this case follows directly from Theorem \ref{th:poisson_characterization} as we only need to shift the argument of the density $p_{\lambda}$ defined in \eqref{eg:poiss_dens}. 

The next theorem for densities on $\R^d$, $d \geq 1$, can be seen as a directional variant of the result known from the theory of infinitely divisible distributions on the real line, see Sgibnev \cite{Sg}, Pakes \cite{Pa1, Pa2}, Shimura and Watanabe \cite{W05}, Watanabe \cite{Wat2}, and Watanabe and Yamamuro \cite{W10, WY}. A similar result has been obtained recently by Watanabe \cite{W20} for densities of infinitely divisible distributions on the half-line. 
Denote: $f_1 := f \wedge 1$.

\begin{theorem}\label{th:inf_div_character}
Let $\nu(dx)=f(x)dx$, where $f$ is an almost radial decreasing function such that \eqref{eq:c1} holds with some $\theta \in \mathds{S}^{d-1}$ and $\gamma \geq 0$.  Then, under each of the above three conditions \textup{(1)--(3)}, the following statements are equivalent.
\begin{itemize}
\item[(a)] $f_1 \in \mathscr{S}_d(\theta, \gamma)$;

\item[(b)] We have
\begin{align} \label{eq:char_exp_id}
\int_{\R^d} e^{\gamma(\theta \cdot y)} f_1(y) dy < \infty,
\end{align}
the expression 
$$
\psi(- i \gamma \theta) =  - \gamma(\theta \cdot b) - \gamma \theta \cdot A \gamma \theta + \int_{\R^d \setminus \left\{0\right\}} 
		\left(1-e^{\gamma(\theta \cdot y)}+\gamma(\theta \cdot y) \I_{B(0,1)}(y) \right) f(y)\, dy
$$ 
makes sense and for every $\lambda \in \Lambda$ and $y \in \R^d$,
\begin{align} \label{eq:char_conv_id}
\lim_{t \to \infty} \frac{p_{\lambda}(t\theta - y)}{\lambda f(t\theta)} = \exp{\left(\gamma(\theta \cdot y) - \lambda \psi(- i \gamma \theta)\right)} \quad \text{and} \quad \limsup_{|x| \to \infty} \frac{p_{\lambda}(x)}{f(x)} < \infty;
\end{align}

\item[(c)] For some (every) $\lambda \in \Lambda$, $p_{\lambda} \in \mathscr{S}_d(\theta, \gamma)$ and $\limsup_{|x| \to \infty} \frac{p_{\lambda}(x)}{f(x)} < \infty$.
\end{itemize}
\end{theorem}

In the theorem above we consider absolutely continuous L\'evy measures with almost radial decreasing densities. 
We refer the reader to \cite[Remark 3(iii)]{W20} for an interesting example of semistable infinitely divisible distribution on $[0,\infty)$ with absolutely continuous L\'evy measure. That distribution has subexponential density, while the density of the corresponding L\'evy measure -- restricted to $(1,\infty)$ and normalized -- is less regular. Moreover, it is proved in \cite{W20} that there are infinitely divisible distributions with subexponential densities on $[0,\infty)$ whose L\'evy measures are not absolutely continuous.

Our results cover the densities of the so-called isotropic unimodal distributions and integrable functions that are uniformly comparable to such densities. Recall that a probability measure $\mu$ on $\R^d$, $d \geq 1$, is called isotropic unimodal if $\mu(dx) = g(|x|)dx$ on $\R^d \setminus \left\{0\right\}$ and $(0,\infty) \ni r \mapsto g(r)$ is a decreasing function \cite[Definition 1.1]{Wat}. This class of probability measures is commonly used in various applications. Just recently, it has received much attention due to the remarkable progress in the potential theory of the isotropic unimodal L\'evy procesess, see Grzywny \cite{G}, Bogdan, Grzywny and Ryznar \cite{BGR1, BGR2}, Kulczycki and Ryznar \cite{KR, KR2}, Grzywny and Kwaśnicki \cite{GK} and references in these papers.

\begin{remark} \label{rem:iso_skip}
Under the assumption that the measures $P_{\lambda}$, $\lambda >0$, are isotropic (in particular, the densities $p_{\lambda}$ are radial functions)
the condition $\limsup_{|x| \to \infty} \frac{p_{\lambda}(x)}{f(x)} < \infty$ can be removed from part (b) of Theorem \ref{th:inf_div_character}. 
Indeed, in such a case this condition is an easy consequence of the convergence in the direction $\theta$ of \eqref{eq:char_conv_id}. 
This observation essentially simplifies and shortens the proof of the implication (a) $\Rightarrow$ (b) in Theorem \ref{th:inf_div_character} for isotropic measures.
\end{remark}

Radial decreasing and almost radial decreasing functions often appear in the literature concerning jump processes and related non-local operators, and their applications to partial differential equations, models of mathematical physics and spectral theory. Typically, they play the role of profiles and majorants for L\'evy kernels and related objects. The most recent progress includes equations involving convolution operators \cite{FT1, FT2}, heat kernels of unbounded pseudodifferential operators \cite{BGR,CGT,KSz2,KSz3,KM,Knop,KnopKul} and the corresponding Schr\"odinger operators \cite{BGJP,KL1,KSch}, structure of the spectrum and properties of eigenfunctions of non-local Schr\"odinger operators \cite{JW,KL2,KL3,KwaM}. Of course, this list of references is far from being complete. Note in passing that the condition \eqref{eq:K0} appeared first in \cite{KL2}, where it turned out to be a very useful tool in proving the fall-off rates for eigenfunctions of non-local Schr\"odinger operators. We expect that due to an increasing interest in models based on exponential functions, the classes $\mathscr{S}_d(\theta, \gamma)$ and our present results also will find applications in some of these areas. 

\smallskip

\textbf{Acknowledgements.} We thank Irmina Czarna, Mateusz Kwa\'snicki, Ren\'e Schilling, Paweł Sztonyk and Toshiro Watanabe for discussions, comments and references. We are grateful to the referee and the editors for careful handling of the paper and helpful comments.

\section{Proof of our main results} \label{sec:equivalence}

For given $f \in L^1(\R^d)$ and $\theta \in \mathds{S}^{d-1}$ we denote
\begin{align} \label{eq:G_theta}
G^f_{\theta}(t,r):= \int_{\substack{|t\theta-y| \textgreater r \\ | y| \textgreater r}}f(t\theta-y)f(y)dy, \quad t > 0, \ \ r >0.
\end{align}

We will need the following lemma. 

\begin{lemma} \label{lem:aux1}
Let $f \in L^1(\R^d)$, $f \geq 0$, be such that there exist $A \geq 1$ and a decreasing function $g:[A,\infty) \to (0,\infty)$ satisfying $f(x) \asymp g(|x|)$, $|x| \geq A$, and let \eqref{eq:c1} and \eqref{eq:c2} hold with some $\theta \in \mathds{S}^{d-1}$ and $\gamma \geq 0$. Then we have the following statements.
\begin{itemize}
\item[(a)] For every fixed $r \geq A$, 
$$
\lim_{t \to \infty} \frac{G^f_{\theta}(t,r)}{{f(t\theta)}} = 2 \int_{|y| \textgreater r} e^{\gamma(\theta \cdot y)}f(y)dy.
$$
\item[(b)]
For every fixed $t_0>1$, 
$$
\lim_{r \to \infty} \sup_{ t \in [1,t_0]} G^f_{\theta}(t,r) = 0.
$$
\end{itemize}
\end{lemma}

\begin{proof}
We first show (a). Fix $r \geq A$ and suppose that $t > 2r$. We have
\begin{align*}
\frac{G^f_{\theta}(t,r)}{{f(t\theta)}} & = \left(\int_{\R^d} \frac{f(t\theta  -y)}{f(t\theta )}f(y)dy - 2 \int_{\R^d} e^{\gamma (\theta \cdot y)} f(y) dy \right) \\ 
& \ \ + 2  \int_{|y| \leqslant r} e^{\gamma (\theta \cdot y)} f(y) dy - \left( \int_{|t\theta - y| \leqslant r}   \frac{f(t\theta  -y)}{f(t\theta )}f(y)dy + \int_{|y| \leqslant r} \frac{f(t\theta  -y)}{f(t\theta )}f(y)dy    \right)   \\ 
& \ \ + 2  \int_{|y| \textgreater r} e^{\gamma (\theta \cdot y)} f(y) dy,
\end{align*}
which leads to
\begin{align*}
\frac{G^f_{\theta}(t,r)}{{f(t\theta)}} & -  2  \int_{|y| \textgreater r} e^{\gamma (\theta \cdot y)} f(y) dy \\
& = \frac{f^{2 \star}(t\theta)}{f(t\theta )} - 2 h(\theta) \\ 
& \ \ + 2  \int_{|y| \leqslant r} e^{\gamma (\theta \cdot y)} f(y) dy - \left( \int_{|t\theta - y| \leqslant r}   \frac{f(t\theta  -y)}{f(t\theta )}f(y)dy + \int_{|y| \leqslant r} \frac{f(t\theta  -y)}{f(t\theta )}f(y)dy    \right).
\end{align*}
By changing variables $y = t \theta - z$ in the second integral in the last line, we get
\begin{align*}
     \frac{G^f_{\theta}(t,r)}{{f(t\theta)}} -  2 & \int_{|y| \textgreater r} e^{\gamma (\theta \cdot y)} f(y) dy \\
		 & = \left( \frac{f^{2 \star}(t \theta)}{f(t \theta)} -  2 h(\theta)\right) + 2  \int_{|y| \leqslant r} \left( e^{\gamma (\theta \cdot y)} - \frac{f(t\theta  -y)}{f(t\theta )} \right) f(y)  dy.
\end{align*}
Due to assumption \eqref{eq:c2} the expression inside the first bracket on the right hand side above converges to zero as $t \to \infty$. The last integral also goes to zero as $t \to \infty$. This is a consequence of \eqref{eq:c1} and the Lebesgue dominated convergence theorem. As $f \in L^1(\R^d)$, it is enough to show that the function 
$$
e^{\gamma (\theta \cdot y)} - \frac{f(t\theta  -y)}{f(\theta t)}
$$ 
is uniformly bounded for $|y|<r$ and sufficiently large $t$'s. Recall that there exist a constant $c \geq 1$ and a decreasing function 
 $g:[A,\infty) \to (0,\infty)$ such that 
\begin{align} \label{eq:profile}
\frac{1}{c} g(|x|) \leq f(x) \leq c g(|x|), \quad |x| \geq A.
\end{align}
Since $t > 2r$ and $|y| < r$ (in particular, $|t \theta| = t \geq A$ and $|t\theta - y| \geq t-r \geq A$), we have
$$
\left|e^{\gamma (\theta \cdot y)} - \frac{f(t\theta  -y)}{f(t\theta )} \right| \leq e^{\gamma r} + c^2 \frac{g(t-r)}{g(t)}.
$$
Let $x_s=s \theta$, $s > 0$. Then $|x_t-x_r| = t-r$, $|x_t|=t$ and, in consequence,
$$
\frac{g(t-r)}{g(t)} = \frac{g(|x_t-x_r|)}{g(|x_t|)} \leq c^2 \frac{f(x_t-x_r)}{f(x_t)}.
$$
By \eqref{eq:c1}, there exists $t_0 > 2r$ such that for $t>t_0$
$$
\frac{f(x_t-x_r)}{f(x_t)} \leq 2e^{\gamma r}.
$$
Hence
$$
\left|e^{\gamma (\theta \cdot y)} - \frac{f(t\theta  -y)}{f(t\theta )} \right| \leq (1+2c^4) e^{\gamma r}, \quad |y| < r, \ \ t > t_0.
$$
This completes the proof of (a).

We now show (b). 
Fix $t_0 > 1$ and observe that for every fixed $r \geq A$ the function $[1,t_0] \ni t \longmapsto G^f_{\theta}(t,r)$ is continuous. This is a consequence of the fact that 
$$
G^f_{\theta}(t,r) = f_r^{2 \star}(t\theta), \quad \text{with} \ \ f_r \in L^1(\R^d) \cap L^{\infty}(\R^d),
$$
where $f_r := f \I_{\left\{|y|>r\right\}}$, see \cite[Theorem 14.8 (ii)]{Sch} (the boundedness of $f_r$ follows from the fact that $f(x)$ has a profile $g(|x|)$ for $|x| \geq A$). Moreover, the function $r \longmapsto G^f_{\theta}(t,r)$ is decreasing and $G^f_{\theta}(t,r) \stackrel{r \to \infty}{\longrightarrow} 0$, for every fixed $t \in [1,t_0]$. It then follows from Dini's theorem that $\lim_{r \to \infty} \sup_{ t \in [1,t_0]} G^f_{\theta}(t,r) = 0$, which completes the proof. 
\end{proof}

We are now in a position to give a proof of our first main result.

\begin{proof}[Proof of Theorem \ref{th:main}]
Let $|x|, r \geq A$. Recall that there exists a decreasing profile $g$ such that \eqref{eq:profile} holds. 
With this we may write
\begin{align*}
\int_{\substack{|x-y| \textgreater r \\ | y| \textgreater r}}\frac{f(x-y)}{f(x)}f(y)dy & \leqslant c^3 \int_{\substack{|x-y| \textgreater r \\ | y| \textgreater r}}\frac{g(|x-y|)}{g(|x|)}g(|y|)dy \\ & = c^3\int_{\substack{\big||x|\theta -y\big| \textgreater r \\ | y| \textgreater r}}\frac{g\big(\big||x|\theta-y\big|\big)}{g(|x|\theta)}g(|y|)dy.
\end{align*}
The last equality follows from the fact that the convolution of two radial functions is again a radial function. By using the comparability of $f$ and $g$ once again, we obtain

$$
\int_{\substack{|x-y| \textgreater r \\ | y| \textgreater r}}\frac{f(x-y)}{f(x)}f(y)dy 
\leqslant c^6 \int_{\substack{\big||x|\theta -y\big| \textgreater r \\ | y| \textgreater r}}\frac{f(|x|\theta-y)}{f(|x|\theta)}f(y)dy.
$$
The function on the right hand side of the above inequality depends on $x$ only via $|x|$. Hence

\begin{equation} \label{eq:Kr_est}
    K^A_f(r) \leqslant c^6 \sup\limits_{t \geq A } \int_{\substack{|t\theta-y| \textgreater r \\ | y| \textgreater r}}\frac{f(t\theta-y)}{f(t\theta)}f(y)dy 
		            = c^6 \sup\limits_{t \geq A } \frac{G^f_{\theta}(t,r)}{{f(t\theta)}}.
\end{equation}
Fix $\varepsilon >0$. By the first part of \eqref{eq:c2}, there exists $r_0 \geq A$ such that 
\begin{equation*}
    0 \leqslant \int_{|y| \textgreater r} e^{\gamma (\theta \cdot y)} f(y) dy\leqslant \frac{\varepsilon}{4c^6}, \quad \text{for} \quad r \geqslant r_0.
\end{equation*}
Moreover, by Lemma \ref{lem:aux1}(a) we can find $t_0>A$ such that for every $t \geq t_0$
$$
\left|\frac{G^f_{\theta}(t,r_0)}{{f(t\theta)}} - 2 \int_{|y| \textgreater r_0} e^{\gamma(\theta \cdot y)}f(y)dy\right| \leq \frac{\varepsilon}{2c^6}.
$$
Using the monotonicity of the function $r \longmapsto G(t,r)$ and these estimates, we get
\begin{align*}
    0 \leqslant \frac{G^f_{\theta}(t,r)}{{f(t\theta)}} & \leqslant \frac{G^f_{\theta}(t,r_0)}{{f(t\theta)}} \\ & = \frac{G^f_{\theta}(t,r_0)}{{f(t\theta)}} - 2 \int_{|y| \textgreater r_0} e^{\gamma(\theta \cdot y)}f(y)dy + 2 \int_{|y| \textgreater r_0} e^{\gamma(\theta \cdot y)}f(y)dy \\
		 & \leqslant  \frac{\varepsilon}{c^6},
\end{align*}
for every $t \geq t_0$ and $r \geq r_0$. This means that 
\begin{equation*}
    \sup\limits_{t \geqslant t_0 }\frac{G^f_{\theta}(t,r)}{{f(t\theta)}}  \leqslant \frac{\varepsilon}{c^6}, \quad \text{for} \quad r \geqslant r_0.
\end{equation*}
On the other hand, it follows from Lemma \ref{lem:aux1}(b) and \eqref{eq:profile} that 
\begin{equation*}
    0 \leq \sup\limits_{t \in [A, t_0]}\frac{G^f_{\theta}(t,r)}{{f(t\theta)}} \leq \frac{c}{{g(t_0)}} \sup\limits_{t \in [1, t_0]}G^f_{\theta}(t,r)  \stackrel{r \to \infty}{\longrightarrow} 0.
\end{equation*}
By \eqref{eq:Kr_est} we may then conclude that
$$
\limsup_{r \to \infty} K^A_f(r) \leq c^6\left(\limsup_{r \to \infty} \sup\limits_{t \in [A, t_0]}\frac{G^f_{\theta}(t,r)}{{f(t\theta)}} +  \limsup_{r \to \infty} \sup\limits_{t  \geq t_0}\frac{G^f_{\theta}(t,r)}{{f(t\theta)}}\right) \leq \varepsilon.
$$
This completes the proof. 
\end{proof}

\section{Applications}

\subsection{Direct applications}
We now give the comparability result which is a multivariate directional version of \cite[Lemma 1.2]{Klup}. We first prove a lemma which is important for our further applications. 

\begin{lemma}\label{lem:comparable}
Let $f_1, f_2 \in L^1(\R^d)$ be such that $f_1, f_2 \geq 0$ and satisfy the following conditions:
\begin{itemize}
\item[(a)] $f_1$ is almost radial decreasing density;
\item[(b)] there exist $A \geq 1$ and a decreasing function $g_2:[A,\infty) \to (0,\infty)$ such that $f_2(x) \asymp g_2(|x|)$, $|x| \geq A$;
\item[(c)] both densities $f_1$ and $f_2$ satisfy the condition \eqref{eq:c1} with some $\theta \in \mathds{S}^{d-1}$ and $\gamma \geq 0$ and
\begin{align} \label{eg:comp_fs}
0 < \liminf_{|x| \to \infty} \frac{f_1(x)}{f_2(x)} \leq \limsup_{|x| \to \infty} \frac{f_1(x)}{f_2(x)} < \infty.
\end{align}
\end{itemize}
If $f_2 \in \mathscr{S}_d(\theta, \gamma)$, then $f_1 \in \mathscr{S}_d(\theta, \gamma)$.
\end{lemma}

\begin{proof}
By the second part of (c), there exists a constant $c \geq 1$ and $r \geq A$ such that
$$
\frac{1}{c} \leq \frac{f_1(x)}{f_2(x)} \leq c, \quad |x| \geq r,
$$
As the densities $f_1$ and $f_2$ have the profiles $g_1$ and $g_2$ which are positive and decreasing, this means that there exists a constant $c_1 \geq c$ such that  
$$
\frac{1}{c_1} \leq \frac{f_1(x)}{f_2(x)} \leq c_1, \quad |x| \geq A,
$$
which implies the comparability $K^A_{f_1}(r) \asymp K^A_{f_2}(r)$, $r \geq A$. Consequently, $\lim_{r \to \infty} K^A_{f_1}(r) = 0$ if and only if $\lim_{r \to \infty} K^A_{f_2}(r) = 0$. 

Since $f_2 \in \mathscr{S}_d(\theta, \gamma)$, we get $\lim_{r \to \infty} K^A_{f_2}(r) = 0$, by Theorem \ref{th:main}. The above observation immediately gives $\lim_{r \to \infty} K^A_{f_1}(r) = 0$. Moreover, for $1 \leq |x| \leq A \leq r$ we have 
$$\int_{|x-y|>r \atop |y| >r} \frac{f_1(x-y)f_1(y)}{f_1(x)} dy  \leq c_2 \int_{|x-y|>r \atop |y| >r} \frac{g_1(|x-y|)g_1(|y|)}{g_1(|x|)} dy \leq c_2 \int_{|y| >r} g_1(|y|) dy,
$$ 
which yields $\lim_{r \to \infty} K_{f_1}(r) = 0$. The assertion follows now directly from Proposition \ref{prop:sufficient}.
\end{proof}

The next proposition is a corollary from Lemma \ref{lem:comparable}

\begin{proposition}\label{prop:comparable}
Let $f_1, f_2 \in L^1(\R^d)$, $f_1, f_2 \geq 0$, be the almost radial decreasing densities such that \eqref{eg:comp_fs} holds.
Assume, in addition, that both $f_1$ and $f_2$ satisfy condition \eqref{eq:c1} with some $\theta \in \mathds{S}^{d-1}$ and $\gamma \geq 0$. Then $f_1 \in \mathscr{S}_d(\theta, \gamma)$ if and only if $f_2 \in \mathscr{S}_d(\theta, \gamma)$.
\end{proposition}

\subsection{Compound Poisson and more general infinitely divisible distributions}

We are now ready to prove our next theorems. 

\begin{proof}[Proof of Theorem \ref{th:poisson_characterization}]
\noindent
(a) $\Rightarrow$ (b):
First note that \eqref{eq:char_exp} holds by definition of $\mathscr{S}_d(\theta, \gamma)$. By Theorem \ref{th:main} we have that $K_f(r) \to 0$ as $r \to \infty$. The convergence in (b) follows then directly from \cite[Lemma 7(a)]{KSz3}. The assumptions (B), (C) in \cite{KSz3} are satisfied as we know that our density $f$ is almost radial decreasing. We use this opportunity to correct the typo in the condition (3) of the assumption (B). One should have $\nu(\left\{x:|x|>r/2\right\})$ instead of $\nu(\left\{x:|x|>r\right\})$ there.

\bigskip

\noindent
(b) $\Rightarrow$ (c):
Fix $\lambda>0$. By \eqref{eq:char_conv}, we have
$$
\lim_{t \to \infty} \frac{p_{\lambda}(t\theta - y)}{p_{\lambda}(t\theta)} 
   = \lim_{t \to \infty} \frac{p_{\lambda}(t\theta - y)}{\lambda f(t\theta)} \frac{\lambda f(t\theta)}{p_{\lambda}(t\theta)} = e^{\gamma(\theta \cdot y)}, \quad y \in \R^d,
$$
which shows that $p_{\lambda}$ satisfies \eqref{eq:c1}. Since $f$ is almost radial decreasing, there exist a decreasing profile function $g:(0,\infty) \to (0,\infty)$ and a constant $c \geq 1$ such that 
\begin{align}\label{eq:poiss_prof}
\frac{1}{c} g(|x|) \leq f(x) \leq c g(|x|), \quad x \in \R^d \setminus \left\{0\right\}.
\end{align}
In particular, for $n \in \N$, 
$$
f^{n \star}(x) \leq c^{2n} f^{n \star}(|x|\theta ), \quad x \in \R^d \setminus \left\{0\right\},
$$
which leads to estimate
\begin{align*}
p_{\lambda}(x) & = e^{-\lambda\left\|f\right\|_1} \sum_{n=1}^{\infty} \frac{\lambda^n}{n!} f^{n \star}(x) \\
               & \leq e^{(c^2-1)\lambda\left\|f\right\|_1} e^{-c^2\lambda\left\|f\right\|_1} \sum_{n=1}^{\infty} \frac{(c^2\lambda)^n}{n!} f^{n \star}(|x|\theta) \\
							 & \leq e^{(c^2-1)\lambda\left\|f\right\|_1} p_{c^2\lambda}(|x|\theta).
\end{align*}
Now, by \eqref{eq:char_conv} with $y = 0$ and \eqref{eq:poiss_prof}, there is a constant $\widetilde c >0$ such that for sufficiently large $|x|$, we get
$$
 p_{c^2\lambda}(|x|\theta) \leq \widetilde c f(|x|\theta) \leq c^2 \widetilde c f(x).
$$
This shows that $\limsup_{|x| \to \infty} \frac{p_{\lambda}(x)}{f(x)} < \infty$ and we are left to show that $p_{\lambda}$ satisfies \eqref{eq:c2}.
The first part of this condition follows from Tonelli's theorem and the fact that $\int e^{\gamma(\theta \cdot y)} f^{n \star}(y)dy = \left(\int e^{\gamma(\theta \cdot y)} f(y)dy\right)^n$, $n \in \N$. Indeed, we have
\begin{align}\label{eq:exp_moment}
\int_{\R^d} e^{\gamma(\theta \cdot y)} p_{\lambda}(y) dy & = e^{-\lambda\left\|f\right\|_1} \sum_{n=1}^{\infty} \frac{\lambda^n}{n!} \int_{\R^d} e^{\gamma(\theta \cdot y)}f^{n \star}(y) dy \nonumber\\
                                                       & = e^{-\lambda\left\|f\right\|_1} \sum_{n=1}^{\infty} \frac{1}{n!} \left(\lambda \int_{\R^d} e^{\gamma(\theta \cdot y)}f(y) dy\right)^n \\
																										 	 & = e^{-\lambda\left\|f\right\|_1} \left(e^{\lambda \int e^{\gamma(\theta \cdot y)} f(y)dy} - 1\right)< \infty. \nonumber
\end{align}
To see the second part of \eqref{eq:c2}, we first notice that by direct calculation we get
 $p_{\lambda}^{2 \star}(x) = p_{2\lambda}(x) - 2e^{-\lambda\left\|f\right\|_1} p_{\lambda}(x)$. Hence,
$$
\frac{p_{\lambda}^{2 \star}(t\theta)}{p_{\lambda}(t\theta)} = 2 \cdot \frac{p_{\lambda}^{2 \star}(t\theta)}{2\lambda f(t\theta)}\frac{\lambda f(t\theta)}{p_{\lambda}(t\theta)} = 2 \cdot \left(\frac{p_{2\lambda}(t\theta)}{2\lambda f(t\theta)}\frac{\lambda f(t\theta)}{p_{\lambda}(t\theta)}- e^{-\lambda \left\|f\right\|_1} \right).
$$
It then follows from \eqref{eq:char_conv} with $y=0$ that
$$
\lim_{t \to \infty} \frac{p_{\lambda}^{2 \star}(t\theta)}{p_{\lambda}(t\theta)} = 2 \left(\exp{\left(\lambda \int \left(e^{\gamma(\theta \cdot y)} -1 \right)f(y)dy\right)} - \exp{(-\lambda \left\|f\right\|_1}) \right).
$$
Using \eqref{eq:exp_moment}, we conclude that
$$
\lim_{t \to \infty} \frac{p_{\lambda}^{2 \star}(t\theta)}{p_{\lambda}(t\theta)} = 2 \int_{\R^d} e^{\gamma(\theta \cdot y)} p_{\lambda}(y) dy,
$$
which completes the proof of part (c).

\bigskip

\noindent
(c) $\Rightarrow$ (a):
By definition of $p_{\lambda}$ and the second part of (c), there exist $A \geq 1$ and a constant $c_1>0$ such that
$$
\lambda e^{-\lambda\left\|f\right\|_1} f(x) \leq p_{\lambda}(x) \leq c_1 f(x), \quad |x| \geq A.
$$
Since $f$ is almost radial decreasing, this means that
$$
p_{\lambda}(x) \asymp g(|x|), \quad |x| \geq A.
$$
Thus it follows from Lemma \ref{lem:comparable} that $f \in \mathscr{S}_d(\theta, \gamma)$, giving (a). 
This completes the proof.

\end{proof}

\begin{proof}[Proof of Theorem \ref{th:inf_div_character}]
\noindent
(a) $\Rightarrow$ (b):
As before, \eqref{eq:char_exp_id} holds by definition of $\mathscr{S}_d(\theta, \gamma)$, and we have $K_{f_1}(r) \to 0$ as $r \to \infty$, by Theorem \ref{th:main}. Moreover, for sufficiently large $r>0$,
$$
\int_{|x-y|>r \atop |y| >r} \frac{f(x-y)f(y)}{f(x)} dy = \int_{|x-y|>r \atop |y| >r} \frac{f_1(x-y)f_1(y)}{f(x)} dy \leq \int_{|x-y|>r \atop |y| >r} \frac{f_1(x-y)f_1(y)}{f_1(x)} dy,
$$
which means that $\lim_{r \to \infty} K_{f}(r) \to 0$ as well.  
The convergence part in \eqref{eq:char_conv_id} follows then from \cite[Theorem 2(a)]{KSz3} in case (1), and from \cite[Theorem 1(a)]{KSz3} in cases (2)--(3) (the rest of the proof of this implication is mainly based on \cite{KSz3} -- in order to avoid repetitions, we just refer precisely to various parts of that paper). As we already explained in the previous proof, assumptions (B) and (C) in the quoted paper are satisfied, and assumption (A) is assumed here a priori. Moreover, the second cited result can be applied to $T= \left\{\lambda\right\}$, for an arbitrary $\lambda \in \Lambda$. Indeed, under the condition $\int_{\R^d} e^{-\lambda_0 \re \phi(\xi)} |\xi| d \xi < \infty$ the assumption (D) in that paper holds true for all such sets $T$, see \cite[Remark 1(e)]{KSz3}. Moreover, the expression $\psi(- i \gamma \theta)$ makes sense by \cite[Theorem 25.17]{Sato} -- this is a consequence of \eqref{eq:char_exp_id}. We only have to show that $\limsup_{|x| \to \infty} \frac{p_{\lambda}(x)}{f(x)} < \infty$ for every fixed $\lambda \in \Lambda$. This is the most technical part of this proof. Recall that if the measures $P_{\lambda}$ are isotropic, it can be omitted, see Remark \ref{rem:iso_skip}.

Consider the cases (2) and (3). We fix $\lambda \in \Lambda$ and first introduce a necessary notation. 
Let $\accentset{\circ}{P}^r_{\lambda}$ and $\bar{P}^r_{\lambda}$, $\lambda, r>0$, be the probability measures given by
$$
\mathcal{F}(\accentset{\circ}{P}^r_{\lambda})(\xi) = \exp\left(-\lambda \int_{B_r(0) \setminus \left\{0\right\}} \left(1 - e^{i \xi \cdot y} + i \xi \cdot y \right)f(y)dy\right), \quad \xi \in \R^d,
$$
$$
\mathcal{F}(\bar{P}^r_{\lambda})(\xi)=\exp\left(-\lambda \int_{B_r(0)^c} \left(1 - e^{i \xi \cdot y}\right)f(y)dy\right), \quad \xi \in \R^d.
$$
Here $B_r(x):= \left\{y \in \R^d: |y-x| < r\right\}$. For any $r>0$ we have 
$$
\accentset{\circ}{P}^r_{\lambda}(dx) = \accentset{\circ}{p}^r_{\lambda}(x)dx \quad \text{with} \quad \accentset{\circ}{p}^r_{\lambda} \in C^1_b(\R^d)
$$
and 
$$
\bar{P}^r_{\lambda}(dx) =  e^{-\lambda \left\|\bar{f_r}\right\|_1}\delta_{0}(dx) + {\bar{p}}^r_{\lambda}(x) dx 
\quad \text{with} \quad
\bar{p}^r_{\lambda}(x) := e^{-\lambda \left\|\bar{f_r}\right\|_1} \sum_{n=1}^\infty \frac{\lambda^n\bar{f_r}^{n*}(x)}{n!},
$$
where $\bar{f_r}(y)= \I_{B_r(0)^c}(y)\,f(y).$ Also, whenever $A \neq 0$, we denote by $G_{\lambda}(dx)$ the Gaussian measure 
which is determined by its Fourier transform $\mathcal{F}(G_{\lambda})(\xi) = \exp\left(-\lambda \xi \cdot A \xi\right)$, $\xi \in \R^d$. 

We now choose $r=r(\lambda) = h(\lambda)$, where $h$ is a scale function defined by \cite[(13)]{KSz3} and put $\accentset{\circ}{p}_{\lambda} := \accentset{\circ}{p}^{r(\lambda)}_{\lambda}, {\bar{p}}_{\lambda} := {\bar{p}}^{r(\lambda)}_{\lambda}$. With this preparation, we can now use the decomposition \cite[(14)--(15)]{KSz3} to get the following representation for the kernel $p_{\lambda}$ (recall that $\lambda \in \Lambda$ is fixed):\ there exists a constant $c=c(\lambda)$ and a vector $z=z(b,\lambda)$ such that
$$
p_{\lambda}(x) = c \sigma_{\lambda}(x-z)+ \int_{\R^d} \bar{p}_{\lambda}(x-z-y) \sigma_{\lambda}(y)dy,
$$
where
$$
\sigma_{\lambda}(x) = \left\{\begin{array}{ll}
\accentset{\circ}{p}_{\lambda}(x) &\mbox{if } A \equiv 0, \\
\accentset{\circ}{p}_{\lambda} \ast G_{\lambda} (x) & \mbox{otherwise} \, .
\end{array}\right.
$$
By \cite[Lemma 1(f) and Lemma 3(b)]{KSz3} and the fact that $f$ is almost radial decreasing, we find $R=R(\lambda) >1$ large enough and $c_1=c_1(\lambda)$, $c_2=c_2(\lambda)$ such that
$$
\bar{p}_{\lambda}(w) \leq c_1 f(w) \quad \text{and} \quad \sigma_{\lambda}(w) \leq c_2 f(w), \quad |w| > R.
$$
Moreover, by \eqref{eq:c1} and the comparability $f(w) \asymp g(|w|)$, $w \neq 0$, there is $c_3>0$ such that 
$$
f(w-y) \leq c_3 f(w), \quad  |w| >2R, \ |y| \leq R. 
$$
Consequently, for every $|x|> 2R +|z|$,
\begin{align*}
p_{\lambda}(x) & = c \sigma_{\lambda}(x-z)+ \left(\int_{|y| \leq R} + \int_{|x-z-y|> R \atop |y| > R} + \int_{|x-z-y|\leq R}\right) \bar{p}_{\lambda}(x-z-y) \sigma_{\lambda}(y)dy\\
               & \leq cc_2 f(x-z) + \left(c_1c_3 \int_{|y| \leq R} \sigma_{\lambda}(y)dy + c_2c_3 \int_{|x-z-y|\leq R} \bar{p}_{\lambda}(x-z-y) dy \right) f(x-z) \\
							                    & \ \   + c_1 c_2 \int_{|x-z-y|> R \atop |y| > R} f(x-z-y) f(y) dy.
\end{align*}
Finally, since $K_f(R) < \infty$, we get 
$$
p_{\lambda}(x) \leq c_4 f(x-z),
$$
and one more use of \eqref{eq:c1} and the comparability with the profile $g$ leads us to the conclusion
$$
p_{\lambda}(x) \leq c_5 f(x).
$$
This completes the proof of (b) in cases (2)--(3). We omit the proof of $\limsup_{|x| \to \infty} \frac{p_{\lambda}(x)}{f(x)} < \infty$ for the case (1) as it is only a modification of the above argument and in fact it is much easier. It combines the decomposition \cite[(16)]{KSz3} with estimates in \cite[Corollary 3(f) and Lemma 3(a)]{KSz3}.

\bigskip

\noindent
(b) $\Rightarrow$ (c):
Fix $\lambda \in \Lambda$. First note that by \cite[Theorem 25.17]{Sato}, the integrability condition \eqref{eq:char_exp_id} gives 
\begin{align} \label{eq:exp_finite}
\int_{\R^d} e^{\gamma(\theta \cdot y)} p_{\lambda}(y) dy < \infty \quad \text{and} \quad \exp{\left( - \lambda \psi(- i \gamma \theta)\right)} = \int_{\R^d} e^{\gamma(\theta \cdot y)} p_{\lambda}(y) dy .
\end{align}
Therefore, by \eqref{eq:char_conv_id} and \eqref{eq:exp_finite},
$$
\lim_{t \to \infty} \frac{p_{\lambda}(t\theta - y)}{p_{\lambda}(t\theta)} 
   = \lim_{t \to \infty} \frac{p_{\lambda}(t\theta - y)}{\lambda f(t\theta)} \frac{\lambda f(t\theta)}{p_{\lambda}(t\theta)} = e^{\gamma(\theta \cdot y)}, \quad y \in \R^d,
$$
and
$$
\frac{p_{\lambda}^{2 \star}(t\theta)}{p_{\lambda}(t\theta)} = 2 \cdot\frac{p_{2\lambda}(t\theta)}{2\lambda f(t\theta)}\frac{\lambda f(t\theta)}{p_{\lambda}(t\theta)} = 2 \exp{\left( - \lambda \psi(- i \gamma \theta)\right)} = 2 \int_{\R^d} e^{\gamma(\theta \cdot y)} p_{\lambda}(y) dy,
$$
which finally implies that $p_{\lambda} \in \mathscr{S}_d(\theta, \gamma)$. The first equality in the last line follows from the fact that $P_{\lambda}$, $\lambda>0$, form a convolution semigroup of measures. This completes the proof of (c).
\bigskip

\noindent
(c) $\Rightarrow$ (a): Suppose (c) is true for some $\lambda \in \Lambda$. We need to show that $\liminf_{|x| \to \infty} \frac{p_{\lambda}(x)}{f(x)} >0$. To this end, we apply a version of the decomposition that was used in the proof of the first implication, by choosing e.g.\ $r =1/2$. We can just follow the argument from the proof of \cite[Lemma 2.6]{KSch} and easily show that in any case (1)--(3) there exist $c_6>0$ such that 
$$
p_{\lambda} (x) \geq c_6 \int_{|y| < 1/2} f(x-z-y) \mu_{\lambda}(dy), \quad |x| >1,
$$
where $z=z(\lambda,b) \in \R^d$ is fixed and $\mu_{\lambda}$ is a probability measure supported by full $\R^d$. 
Now, by using \eqref{eq:c1} and the comparability $f(w) \asymp g(|w|)$, $w \neq 0$, we can find $c_7 >0$ and $R>1$ large enough such that $f(x-z-y) \geq c_7 f(x)$, $|y| < 1/2$, $|x| >R$. This implies that
$$
p_{\lambda} (x) \geq c_5 f(x), \quad |x| >R,
$$
because $\mu_{\lambda}(B(0,1/2)) >0$. Since we already know from (c) that $\limsup_{|x| \to \infty} \frac{p_{\lambda}(x)}{f(x)} < \infty$, it gives that there exists $A > 1$ such that
$$
p_{\lambda}(x) \asymp g(|x|), \quad |x| \geq A.
$$
In consequence, we can apply Lemma \ref{lem:comparable} to show that $f \in \mathscr{S}_d(\theta, \gamma)$. This yields (a) and completes the proof of the lemma.

\end{proof}

\subsection{Examples} \label{sec:ex}

We now illustrate our results by discussing several specific examples of densities. We give positive and negative ones.  

Typical examples of almost decreasing densities are functions $f \in L^1(\R^d)$ which can be represented as
\begin{align}\label{eq:typical_f}
f(x) = \eta\left(\frac{x}{|x|}\right)g(|x|), \quad x \neq 0,
\end{align}
where $\eta : \mathds{S}^{d-1} \rightarrow [c_1, c_2]$, for some $0<c_1 \leq c_2 < \infty$, and $g:(0, \infty) \rightarrow (0, \infty)$ is a decreasing function. Clearly, in this case the function $g$ is a profile of the density $f$. If $\eta$ is continuous in $\theta \in \mathds{S}^{d-1}$, then
\begin{align} \label{eq:conv_eta}
\eta\left(\frac{t\theta - y}{|t\theta - y|}\right) \to \eta(\theta) \quad \text{as} \ \ t \to \infty, \ \ \text{for every} \ y \in \R^d.
\end{align} 
Moreover, when $\eta \equiv const.$, then the density $f$ is radial. Throughout this section, we assume that \eqref{eq:typical_f} and \eqref{eq:conv_eta} hold.

The profile $g$ may have a fairly general shape. Here we consider only two specific cases which are common in applications.

\begin{example} \label{ex:ex1} Let 
$$
g(r) = (1 \vee r)^{-\beta}, \quad \beta \geq 0. 
$$
Clearly, $f \in L^1(\R^d)$ if and only if $\beta > d$. For every fixed $\theta \in \mathds{S}^{d-1}$ and $y \in \R^d$, we have
\begin{align} \label{eq:first_polyn}
\lim_{t \to \infty} \frac{g(|t\theta-y|)}{g(|t\theta|)} = \lim_{t \to \infty} \frac{|t\theta-y|^{-\beta}}{t^{-\beta}} = \lim_{t \to \infty} \frac{1}{|\theta - y/t|^{\beta}} = 1.
\end{align}
Therefore, if \eqref{eq:conv_eta} is true for a given $\theta \in \mathds{S}^{d-1}$, then 
$$
\lim_{t \to \infty} \frac{f(t\theta-y)}{f(t\theta)} = 1, \quad \text{for every} \ \ y \in \R^d, 
$$
i.e., the condition \eqref{eq:c1} holds for such $\theta$ with $\gamma = 0$. It then makes sense to ask about the sub-exponentiality of the density $f$ in the direction $\theta$. We have to verify \eqref{eq:c2} (with $h \equiv \left\|f\right\|_1$).

As shown in Section \ref{sec:equivalence}, we are left to check the convergence $\lim_{r \to \infty}K_f(r) = 0$. Let $r>1$ and $|x| \geq 1$. We have
$$
\int_{|y-x|> r \atop |y|>r} \frac{f(x-y)}{f(x)}f(y) dy \leq \frac{c_2^2}{c_1} \int_{|y-x|> r \atop |y|>r} \left(\frac{|x|}{|x-y||y|}\right)^{\beta}dy 
$$
and, since $y$ and $x-y$ play a symmetric role, the integral on the right hand side is equal to
$$
2 \int_{|y-x|> r, \ |y|>r \atop |y| < |y-x|} \left(\frac{|x|}{|x-y||y|}\right)^{\beta} dy \leq 2 \int_{|y-x|> r, \ |y|>r \atop |y| < |y-x|} \left(\frac{|x|}{(|x|/2)|y|}\right)^{\beta} dy \leq 2^{1+\beta} \int_{|y|>r} |y|^{-\beta} dy
$$
(clearly, this extends to every profile $g$ having the doubling property $g(2r) \leq c g(r)$, $r >0$). 
This shows that $K_f(r) \leq 2^{1+\beta} (c_2^2/c_1) \int_{|y|>r} |y|^{-\beta} dy$. In particular, $\lim_{r \to \infty}K_f(r) = 0$ for the whole range of $\beta >d$. 
Consequently, $f \in \mathscr{S}_d(\theta, 0)$ and all assertions of our Theorem \ref{th:poisson_characterization} apply. 
\end{example}

\begin{example} \label{ex:ex2} Let 
 \begin{align} \label{eq:g_exp}
g(r) = e^{-m r} (1 \vee r)^{-\beta}, \quad \text{with} \quad m > 0, \ \beta \geq 0. 
 \end{align}
Clearly, due to exponential tempering, we now have $f \in L^1(\R^d)$ for the whole range of parameters above.

We first show that if \eqref{eq:conv_eta} is true for a given $\theta \in \mathds{S}^{d-1}$, then \eqref{eq:c1} holds for such $\theta$ with $\gamma = m$.
Let $\theta \in \mathds{S}^{d-1}$ and $y \in \R^d$. For $t>0$, 
    \begin{align}
        |t\theta|-|t\theta-y|
				= \frac{\sum_{i=1}^{d} (t\theta_i)^2  - \sum_{i=1}^{d} (t\theta_i - y_i)^2}{\sqrt{\sum_{i=1}^{d} (t\theta_i)^2} + \sqrt{\sum_{i=1}^{d} (t\theta_i - y_i)^2}}
				= \frac{2 (\theta \cdot y) -\frac{1}{t} \sum_{i=1}^{d}y_{i}^{2} }{|\theta| + \sqrt{\sum_{i=1}^{d} (\theta_{i} - \frac{y_i}{t})^{2}}}.
    \end{align}
It then follows from this calculation and \eqref{eq:first_polyn} that if \eqref{eq:conv_eta} is true for some $\theta \in \mathds{S}^{d-1}$, then
\begin{align*}
\lim_{t \to \infty} \frac{f(t\theta-y)}{f(t\theta)} 
      = \lim_{t \to \infty} \frac{\eta\left(\frac{t\theta - y}{|t\theta - y|}\right)|t\theta-y|^{-\beta}}{\eta(\theta)t^{-\beta}} e^{m( |t\theta|-|t\theta-y|)} = e^{m(\theta \cdot y)} , \quad \text{for every} \ \ y \in \R^d.
\end{align*}
This means that the condition \eqref{eq:c1} holds for such $\theta$ with $\gamma = m$, for all $\beta \geq 0$. 

We conclude this example by proving the following result. 

\end{example}
		
\begin{proposition} \label{prop:ex2} One has:
$$
f \in \mathscr{S}_d(\theta, m) \quad \Longleftrightarrow \quad \beta > \frac{d+1}{2}.
$$
\end{proposition}

\begin{proof}
We only need to show that \eqref{eq:c2} holds if and only if $\beta > (d+1)/2$. 
By Corollary \ref{cor:characterization}, this can be reduced to proving the equivalence:
$$
K_f(r) \searrow 0 \quad \Longleftrightarrow \quad \beta > \frac{d+1}{2}. 
$$
In fact, we can see more -- a conclusion follows from Lemma \ref{lem:final} proven below. 
\end{proof}

\begin{lemma} \label{lem:final}
Let $f$ be as in \eqref{eq:typical_f} with $g$ given by \eqref{eq:g_exp}. Then for every $r >1$ we have
\begin{align} \label{eq:K_f_bounds}
K_f(r)
    & 
    \begin{cases}
        \asymp r^{\frac{d+1}{2} - \beta} &\text{if\ \ } \beta > \frac{d+1}{2},\\
        = \infty &\text{if\ \ } 0 \leq \beta \leq \frac{d+1}{2}.
    \end{cases}
\end{align}
\end{lemma}

\begin{proof}

We first establish the upper bound for $\beta > (d+1)/2$. We have
\begin{align*}
\int_{|y-x|> r \atop |y|>r} \frac{f(x-y)}{f(x)}f(y) dy & \leq \frac{c_2^2}{c_1} \int_{|y-x|> r \atop |y|>r} \frac{g(|x-y|)}{g(|x|)}g(|y|)dy.
\end{align*}
Denote the integral on the right hand side of the above inequality by $I:=I(x,r)$. Clearly, we only need to find the upper bound for $I$.
Suppose first that $1 \leq |x| \leq r$. Then, by monotonicity of $g$,  
$$
I   \leq \frac{g(r)}{g(|x|)} \int_{|y-x|> r \atop |y|>r} g(|y|)dy \leq \int_{|y|>r} g(|y|)dy.
$$
Let now $|x| > r$. Since $y$ and $x-y$ play a symmetric role, we can observe that
\begin{align*}
I & \leq \int_{r < |y-x| < |x| \atop r < |y| < |x|} \frac{g(|x-y|)}{g(|x|)}g(|y|)dy 
                                              + 2 \int_{|y-x| \geq |x| \atop |y| > r} \frac{g(|x-y|)}{g(|x|)}g(|y|)dy \\
	& \leq  2\int_{r < |y-x| < |x|, \ r < |y| < |x| \atop |y| < |x-y|} \frac{g(|x-y|)}{g(|x|)}g(|y|)dy
                                                         + 2 \int_{|y| > r}g(|y|)dy.
\end{align*}
Denote by $J:=J(x,r)$ the first integral on the right hand side. For $d=1$ the exponential terms under this integral cancels and we can follow the estimates from the second part of Example \ref{ex:ex1}, getting the desired bound. Therefore, we are left to consider $d \geq 2$ only. We see that the function $J(\cdot,r)$ is rotation invariant, i.e.\ $J = J(|x|,r)$. In order to simplify our further calculations, we may then assume that $x=(x_1,0,\ldots,0)$ with $x_1>r$. By direct calculation, we can check that $0 \leq |x|-|x-y| \leq y_1$ on the domain of integration of $J$. Hence
\begin{align*}
J & = \int_{r < |y-x| < |x|, \ r < |y| < |x| \atop |y| < |x-y|} e^{m(|x|-|x-y|)} \left(\frac{|x|}{|x-y|}\right)^{\beta} g(|y|)dy \\
  & \leq \int_{r < |y-x| < |x| \atop r < |y| < |x|} \left(\frac{|x|}{(|x|/2)}\right)^{\beta} e^{m(y_1-|y|)}|y|^{-\beta}dy 
	\leq 2^{\beta} \int_{|y| > r \atop y_1 > 0} e^{m(y_1-|y|)}|y|^{-\beta}dy
\end{align*}
Now, by using the hyper-spherical coordinates or a modification of the argument from \cite[p.\ 61]{KSch}, we can show that \
$$
J \leq c_3 \int_r^{\infty} y_1^{\frac{d-1}{2}-\beta} dy_1 = c_4 r^{\frac{d+1}{2}-\beta}.
$$
By collecting all the estimates above and by comparing $\int_{|y|>r} g(|y|)dy = \int_r^{\infty} e^{-ms} s^{d-1-\beta} ds$ with $r^{\frac{d+1}{2}-\beta}$, we conclude that the upper bound in \eqref{eq:K_f_bounds} holds true.

We now give the lower bound. Let $x_n = (2n,0,\ldots,0)$, $n > r$. By following directly the estimates in \cite[pages 382--383]{KSz2}, we can show that 
\begin{align*}
\int_{|y-x_n|> r \atop |y|>r} \frac{f(x_n-y)}{f(x_n)}f(y) dy & \geq \frac{c_1^2}{c_2} \int_{|y-x_n|> r \atop |y|>r} \frac{g(|x_n-y|)}{g(|x_n|)}g(|y|)dy \\
                         & \geq \frac{c_1^2}{c_2} \int_{r < y_1 < n} \frac{g(|x_n-y|)}{g(|x_n|)}g(|y|)dy \geq c_5 \int_r^{n} y_1^{\frac{d-1}{2}-\beta} dy_1.
\end{align*}
Consequently, 
$$
K_f(r) \geq c_5 \int_r^{n} y_1^{\frac{d-1}{2}-\beta} dy_1, \quad n > r,
$$
which implies that $K_f(r) \geq c_6 r^{\frac{d+1}{2}-\beta}$ for $\beta > (d+1)/2$, and $K_f(r) = \infty$ for $ \beta \in [0,(d+1)/2]$. This completes the proof of the lemma.
\end{proof}

\end{document}